\documentclass[12pt]{amsart}
\usepackage{times}
\usepackage[T1]{fontenc}
\usepackage{dsfont}
\usepackage{mathrsfs}
\usepackage[colorlinks]{hyperref}
\usepackage{xcolor}
\usepackage[a4paper,asymmetric]{geometry}
\usepackage{mathscinet}
\usepackage{latexsym}
\usepackage{amsthm}
\usepackage{amssymb}
\usepackage{amsfonts}
\usepackage{amsmath}
\newtheorem{theorem}{Theorem}[section]
\newtheorem{thm}[theorem]{Theorem}

\newtheorem{lem}[theorem]{Lemma}
\newtheorem{proposition}[theorem]{Proposition}

\newtheorem{corollary}[theorem]{Corollary}

\theoremstyle{definition}

\newtheorem{defn}[theorem]{Definition}

\theoremstyle{remark}

\newtheorem{rem}[theorem]{Remark}
\numberwithin{equation}{section}

 \DeclareMathAlphabet{\mathpzc}{OT1}{pzc}{m}{it}

\def\re{{\mathrm e}}

 \newcommand{\E}{\mathbb{E}}            
 
 \newcommand{\e}{\varepsilon}

 \newcommand{\PP}{\mathbb{P}}

 \newcommand{\Be}{\begin{equation}}
 \newcommand{\Ee}{\end{equation}}
 \newcommand{\Bs}{\begin{split}}
 \newcommand{\Es}{\end{split}}
  \newcommand{\Bes}{\begin{equation*}}
 \newcommand{\Ees}{\end{equation*}}
 \newcommand{\BT}{\begin{thm}}
 \newcommand{\ET}{\end{thm}}
 \newcommand{\Bp}{\begin{proof}}
 \newcommand{\Ep}{\end{proof}}
 \newcommand{\BL}{\begin{lem}}
 \newcommand{\EL}{\end{lem}}
 \newcommand{\BP}{\begin{proposition}}
 \newcommand{\EP}{\end{proposition}}
 \newcommand{\BC}{\begin{corollary}}
 \newcommand{\EC}{\end{corollary}}
 \newcommand{\BR}{\begin{rem}}
 \newcommand{\ER}{\end{rem}}
 \newcommand{\BD}{\begin{defn}}
 \newcommand{\ED}{\end{defn}}
 \newcommand{\BI}{\begin{itemize}}
 \newcommand{\EI}{\end{itemize}}

  \newcommand{\dif}{{\rm d}}

\def\PP{\mathbb P}

\def\<{\left<}\def\>{\right>}
\def\({\left(}\def\){\right)}

\begin{document}
\title
{Singular integrals of stable subordinator}

\author[L. Xu]{Lihu Xu}
\address{1. Department of Mathematics, Faculty of Science and Technology,  University of  Macau, Taipa, Macau. 2. UM Zhuhai Research Institute, Zhuhai, China}
\email{lihuxu@umac.mo}

\maketitle
\begin{minipage}{140mm}
\begin{center}
{\bf Abstract}
\end{center}
It is well known that $\int_{0}^{1} t^{-\theta} \dif t<\infty$ for $\theta \in (0,1)$ and $\int_{0}^{1} t^{-\theta} \dif t=\infty$ for $\theta \in [1,\infty)$. Since $t$ can be taken as an $\alpha$-stable subordinator with $\alpha=1$, it is natural to ask whether $\int_{0}^{1} t^{-\theta} d S_{t}$ has a similar property when $S_{t}$ is an $\alpha$-stable subordinator with $\alpha \in (0,1)$. We show that $\theta=\frac 1\alpha$ is the border line such that $\int_{0}^{1} t^{-\theta} d S_{t}$ is finite a.s. for $\theta \in (0, \frac 1\alpha)$ and blows up a.s. for $\theta \in [\frac1\alpha,\infty)$. When $\alpha=1$, our result recovers that of $\int_{0}^{1} t^{-\theta} \dif t$. Moreover, we give a $p$-th moment estimate for the integral when $\theta \in (0,\frac 1\alpha)$.
\end{minipage}

\vspace{4mm}

\medskip
\noindent
{\bf Keywords}:  $\alpha$-stable subordinator, singular integral of $\alpha$-stable subordinator.

\medskip
\noindent


\section{Main result}
It is well known that $\int_{0}^{1} t^{-\theta} \dif t<\infty$ for $\theta \in (0,1)$ and $\int_{0}^{1} t^{-\theta} \dif t=\infty$ for $\theta \in [1,\infty)$. Since $t$ is an $\alpha$-stable subordinator with $\alpha=1$, it is natural to ask whether $\int_{0}^{1} t^{-\theta} d S_{t}$ has a similar property when $S_{t}$ is an $\alpha$-stable subordinator with $\alpha \in (0,1)$. We show that $\theta=\frac 1\alpha$ is the border line such that the integral is finite for $\theta \in (0, \frac 1\alpha)$ and blows up for $\theta \in [1,\infty)$. When $\alpha=1$, our result recovers that of $\int_{0}^{1} t^{-\theta} \dif t$.

Recall that $\alpha$-stable subordinator $S_t$ with $\alpha \in (0,1)$ is a non-negative real valued increasing  L\'evy process with the Laplcace transform
\ \ \
$$\E \re^{-\lambda S_t}=\exp\left(-\lambda^{\alpha} t\right), \ \ \ \forall \ \ \lambda>0$$
where $\alpha \in (0,1)$ and the intensity L\'evy measure is
\ \ \
$$\nu(dy)=c_\alpha y^{-\frac \alpha2-1}1_{(0,\infty)}(y) \dif y$$
where $c_\alpha=\frac{\alpha}{\Gamma(1-\alpha)}$, see \cite[p. 73]{Ber96}. Note that $S_{t}=t$ when $\alpha=1$.  \\

Our main theorem is
\begin{thm} \label{t:GamInt}
Let $S_t$ be a stable subordinator of index $\alpha \in (0,1)$.

\noindent (1). If $\theta \in [\frac 1\alpha,\infty)$, we have
\Be   \label{e:EllInt1-0}
\int_0^T t^{-\theta} \dif S_t=\infty \ \ \ \ \forall \ \ T>0, \ \ \ \ \ \ \ a.s..
\Ee
(2). If $\theta \in (0,\frac 1{\alpha})$,  
Then
\Be   \label{e:EllInt1-1}
\int_0^T t^{-\theta} \dif S_t<\infty \ \ \ \ \forall \ \ T>0, \ \ \ \ \ \ \ a.s..
\Ee
(3). If $\theta \in (0,\frac 1\alpha)$, for any $p \in (0,\alpha)$ we have
\Be
\E \left(\int_0^T t^{-\theta} \dif S_t\right)^{p} \le \left(\frac{2^{\frac{p}\alpha+p\theta} \ \E S^{p}_1}{2^{\frac{p}\alpha}-2^{p\theta}}+1\right)  T^{(\frac 1 \alpha-\theta)p}, \ \ \ \ \ \ T>0.
\Ee
(4). For any $\lambda>0$ and $p \in (0,\alpha)$, we have
 \Be
\E\left(\int_0^T \re^{-\lambda (t-s)} \dif S_t\right)^{p} \le  \frac{\re^{p\lambda}}{\re^{p \lambda}-1} \E S^{p}_1, \ \ \ \ \ \ \ \ \ \ T>0.
\Ee
\end{thm}

\section{Proof of the main result}
It is well known that an $\alpha$-stable subordinator $S_t$ with $\alpha \in (0,1)$ is a pure jump process which is c\`adl\`ag and
strictly increasing almost surely, i.e.,
\
\Be \label{e:StrInc}
\PP \left(S_t {\rm \ is \ c\grave{a}dl\grave{a}g\ and \ strictly \ increasing \ and \ only \ has \ pure \ jumps.}\right)=1.
\Ee
Indeed, a stable type process only has pure jumps and c\`adl\`ag trajectories, by \cite[Lemma 2.1]{Zhang14}, with probability 1, the jump points of $S_t$ are dense. So, \eqref{e:StrInc} holds.
\vskip 2mm

Let us now prove Theorem \ref{t:GamInt}, to this end, we first show the following lemma.
\begin{lem} \label{l:GamInt}
Let $S_t$ be a stable subordinator of index ${\alpha} \in (0,1)$.
(1). For $\theta \in [\frac 1\alpha,\infty)$, we have
\Be   \label{e:EllInt-0}
\int_0^T t^{-\theta} \dif S_t=\infty \ \ \ \ \forall \ \ T>0, \ \ \ \ \ \ \ a.s..
\Ee
(2). For $\theta \in (0,\frac 1\alpha)$, we have
 \Be   \label{e:EllInt-0}
\int_0^T t^{-\theta} \dif S_t<\infty \ \ \ \ \forall \ \ T>0, \ \ \ \ \ \ \ a.s.
\Ee
and
\Be   \label{e:EllInt1}
\int_0^T t^{-\theta} \dif S_t=T^{-\theta} S_T+\theta \int_0^T  t^{-\theta-1} S_t\dif t \ \ \ \ \forall \ \ T>0, \ \ \ \ \ \ \ a.s..
\Ee
\end{lem}
\begin{proof}
Let $\ell_t$ be a path of $S_t$. By \eqref{e:StrInc}, with probability 1, $\ell_t$ is c$\grave{a}$dl$\grave{a}$g and strictly increasing and only has pure jumps.

(1). When $\theta \ge \frac 1\alpha$, let $0<\e<T$, we have
\Bes
 \int_0^T t^{-\theta} \dif \ell_t\ \ge \ \int_{0}^{\e}  t^{-\theta} \dif \ell_t \ \ge \ \e^{-\theta} \ell_{\e}.
\Ees
By \cite[Proposition 10]{Ber96}, take $h(x)=x^{\theta}$ and note $\Phi(x)=x^{\alpha}$ therein, we have
$$\int_{0}^{1} \Phi\left(\frac1{h(x)}\right) \dif x\ = \ \int_{0}^{1} (x^{-\theta})^{\alpha} \dif x\ = \ \infty.$$
Hence, with probability $1$, we have
$$\limsup_{\e \rightarrow 0+} \e^{-\theta} \ell_{\e} \ = \ \infty,$$
which implies that
\Bes
 \int_0^T t^{-\theta} \dif \ell_t\ = \ \infty.
\Ees
(2). Using integration by parts for  Stieltjes integral \cite{Hew60}, for any small $\epsilon>0$,
\
\Be \label{e:ElHe0-0}
 \int_\e^T t^{-\theta} \dif \ell_t
 =\ell_{T} T^{-\theta}- \e^{-\theta} \ell_\e+\theta \int_\e^T t^{-\theta-1} \ell_t \dif t
\Ee
Taking $h(t)=t^{\frac1 {\alpha}-\delta}$ with some $0<\delta < \frac 1{\alpha}-\theta$, we have $0<\theta<\frac 1\alpha-\delta$ and
\ \ \ \
 $$\int_0^1 h^{-\alpha}(t) \dif t\ <\ \infty.$$
 By \cite[Proposition 10]{Ber96}, with probability $1$, we have
 $$\lim_{t \downarrow 0} \frac{\ell_t}{h(t)}=0,$$
and thus
\
\Be \label{e:ElHe1}
\lim_{\e \downarrow 0} \e^{-\theta}\ell_\e \le \lim_{\e \downarrow 0} \frac{\ell_\e}{h(\e)}=0.
\Ee
Moreover, with probability 1, for  $t \in (0,1)$ we have
\
\Bes
\frac{\ell_t}{t^{\theta+1}}=\frac{\ell_t}{h(t)} t^{-1+(\frac{1}{\alpha}-\delta-\theta)} \le C t^{-1+(\frac{1}{\alpha}-\delta-\theta)}.
\Ees
Since $t^{-1+(\frac{1}{\alpha}-\delta-\theta)}$ is integrable over $(0,1)$, with probability 1, we have
\
\Be \label{e:ElHe2}
\lim_{\e \downarrow 0}  \int_\e^{T} t^{-\theta-1} \ell_t \dif t<\infty.
\Ee
We may take $\e \to 0+$ in \eqref{e:ElHe0-0} and get
\ \ \
\Be \label{e:ElHe0}
 \int_0^T t^{-\theta} \dif \ell_t
 =\ell_{T} T^{-\theta}+\theta \int_0^T t^{-\theta-1} \ell_t \dif t \ \ \  {\rm with \ probability \ 1}.
\Ee
Since the above argument holds for the sample path $\ell_t$ (of subordinator $S_t$) with probability $1$, (1)-(2) of the lemma hold.
\end{proof}




\begin{proof}[Proof of Theorem \ref{t:GamInt}]
The relations in (1) and (2) have been proved, it remains to show that the estimates in (3) and (4) hold.

Since $p \in (0,\alpha)$, we have $p \in (0,1)$. By \eqref{e:EllInt1} and the easy inequality $(a+b)^{p} \le a^{p}+b^{p}$ for $a, b \ge 0$, we have
\
\
\Bes
\begin{split}
\E \left(\int_0^T t^{-\theta} \dif S_t\right)^{p}  \le  \E \left(\frac{S_T}{T^\theta}\right)^{p}+
 \E \left(\int_0^T S_t t^{-\theta-1} \dif t\right)^{p}
\end{split}
\Ees
By the scaling property $\E S^{p}_T=T^{\frac{p} \alpha} \E S^{p}_1$, we have
\
\
\Bes
\E \left(\frac{S_T}{T^\theta}\right)^{p} \le T^{\frac{p}{\alpha}-{p \theta}}\E S_{1}^{p}.
\Ees
For any small $\e>0$, as $n$ is sufficiently large so that $\frac{T}{2^{n+1}} \le \e$, we have
\
\Bes
\begin{split}
&\E \left(\int_\e^T \frac{S_{r}}{r^{\theta+1}} \dif r\right)^{p}   \ \le \
\E \left(\sum_{k=0}^n \int_{\frac T{2^{k+1}}}^{\frac T{2^{k}}} \frac{S_{r}}{r^{\theta+1}} \dif r\right)^{p}   \\
& \ \le \ \sum_{k=0}^n \E \left(\int_{\frac T{2^{k+1}}}^{\frac T{2^{k}}} \frac{S_{r}}{r^{\theta+1}} \dif r\right)^{p}
 \ \le \ \sum_{k=0}^n \E \left(\int_{\frac T{2^{k+1}}}^{\frac T{2^{k}}} \left(\frac T{2^{k+1}}\right)^{-\theta-1}{S_{\frac T{2^{k}}}} \dif r\right)^{p} \\
&\ =\ \left(\frac 2T\right)^{\theta p} \sum_{k=0}^\infty  2^{k\theta p} \E S^{p}_{\frac T{2^k}}.
\end{split}
\Ees
By the scaling property $\E S^{p}_t=t^{\frac{p} \alpha} \E S^{p}_1$ again, we obtain
\
\Be   \label{e:SS-3}
\begin{split}
\E \left(\int_\e^T \frac{S_{r}}{r^{\theta+1}} \dif r\right)^{p} \le \E S^{p}_1 \left(\frac 2T\right)^{\theta p} T^{\frac p{\alpha}} \sum_{k=0}^\infty  2^{kp(\theta-\frac 1\alpha)},
\end{split}
\Ee
which immediately gives the bound in (3), as desired.

For the second bound, by the scaling property $\E S^{p}_t=t^{\frac{p} \alpha} \E S^{p}_1$, we have
\Bes
\begin{split}
&\E\left(\int_0^T \re^{-\lambda(T-t)} \dif S_t\right)^{p}  \le \E \left(\sum_{k=0}^{[T]}\int_{k}^{k+1} \re^{-\lambda (T-t)} \dif S_t\right)^{p}  \\
& \le \sum_{k=0}^{[T]} \E\left(\int_{k}^{k+1} \re^{-\lambda(T-t)} \dif S_t\right)^{p}
 \le \sum_{k=0}^{[T]}  \re^{-p\lambda(T-k-1)} \E\left(S_{k+1}-S_k\right)^{p} \\
& =\E S^{p}_1 \sum_{k=0}^{[T]}  \re^{-p\lambda(T-k-1)}
 \le \E S^{p}_1  \sum_{k=0}^{\infty}  \re^{-p k \lambda},
\end{split}
\Ees
which implies the estimate in (4) immediately.
The proof is completed.
\end{proof}

\noindent{\bf Acknowledgments}:  This research is supported by the following grants:  NNSFC 11571390, Macao S.A.R. (FDCT 038/2017/A1, FDCT 030/2016/A1), University of Macau MYRG 2016-00025-FST.

\bibliographystyle{amsplain}

\end{document}